\documentclass[letterpaper, 10 pt, conference]{ieeeconf}  % Comment this line out
\IEEEoverridecommandlockouts                              % This command is only
\include{references.bib}
\usepackage{cite}
\usepackage{graphics} % for pdf, bitmapped graphics files
\usepackage{epsfig} % for postscript graphics files
\usepackage{mathptmx} % assumes new font selection scheme installed
\usepackage{times} % assumes new font selection scheme installed
\usepackage{amsmath} % assumes amsmath package installed
\usepackage{amssymb}  % assumes amsmath package installed
\usepackage{cite}
\usepackage{url}
\usepackage{multirow}
\usepackage{color}

\providecommand{\com[2]}{\begin{tt}[#1: #2]\end{tt}}

\newtheorem{remark}{Remark}
\newtheorem{ex}{Example}

\newcommand{\conitope}{{\mbox{\rm Coni}}}
\newcommand{\coni}{{\mbox{\rm Coni}}}

\newcommand{\co}{{\mathbb C}}

\newcommand{\re}{{\mathbb R}}
\newcommand{\C}{{\mathbb C}}
\newcommand{\rn}{{\mathbb R^n}}

\newcommand{\cn}{{\mathbb C^n}}
\newcommand{\real}{{\rm Re }}

\newcommand{\cS}{{\cal{S}}}

\newcommand{\cP}{{\cal{P}}}
\newcommand{\cM}{{\cal{M}}}
\newcommand{\cm}{{\cal{M}}}
\newcommand{\cK}{{{K}}}

\newcommand{\conv}{{{\rm Conv} }}
\newcommand{\inter}{{{\rm int} }}

\newcommand{\RE}{{\rm Re}}
\newcommand{\p}{\prime}

\def\PF{{\mathcal P}({\cal M})}
\def\nxn{n\! \times \!n\,}
\newcommand{\N}{{\mathbb N}}
\def\P{{\mathcal P}}

\providecommand{\rmj[1]}{{\color{blue}#1}}

\newtheorem{thm}{Theorem}
\newtheorem{proposition}{Proposition}
\newtheorem{definition}{Definition}

\newtheorem{lemma}{Lemma}

\title{\LARGE \bf
Lifted polytope methods for stability analysis of switching systems}

\author{Rapha\"el~M. Jungers %\thanks{Division of Applied Mathematics,
\thanks{ICTEAM Institute, Universit\'e catholique de Louvain, 4 avenue Georges Lemaitre,
B-1348 Louvain-la-Neuve, Belgium. R.J. is supported by the "Communaut\'e francaise de Belgique - Actions de Recherche Concert\'ees", and by the Belgian Programme on Interuniversity Attraction Poles initiated by the Belgian Federal Science Policy Office.  R.J. is a F.R.S.-FNRS fellow. {\tt\small raphael.jungers@uclouvain.be}.}
\and Nicola Guglielmi\thanks{Dipartimento di Matematica Pura e Applicata,
       Universit\`a di L'Aquila, via Vetoio,
       67010 L'Aquila, Italy.
       {\tt\small guglielm@univaq.it}.}
\and Antonio Cicone \thanks{Dipartimento di Matematica Pura e Applicata,
       Universit\`a di L'Aquila, via Vetoio,
       67010 L'Aquila, Italy.
       {\tt\small antonio.cicone@univaq.it}.}}

\begin{document}

\maketitle
\thispagestyle{empty}
\pagestyle{empty}

%%%%%%%%%%%%%%%%%%%%%%%%%%%%%%%%%%%%%%%%%%%%%%%%%%%%%%%%%%%%%%%%%%%%%%%%%%%%%%%%

\begin{abstract}We describe new methods for deciding the stability of switching systems.  The methods build on two ideas previously appeared in the literature: the polytope norm iterative construction, and the lifting procedure.  Moreover, the combination of these two ideas allows us to introduce a pruning algorithm which can importantly reduce the computational burden.

We prove several appealing theoretical properties of our methods like a finiteness computational result which extends a known result for unlifted sets of matrices, and provide numerical examples of their good behaviour. \end{abstract}

\section{Introduction}

A \emph{switching discrete time linear dynamical system} in $\mathbb R^n$ is defined by 
 \begin{eqnarray}\label{switching-system}
 x_{t+1}&=& A_{\sigma(t)}x_t \quad A_{\sigma(t)}\in \cM\\
\nonumber x_0 & \in & \mathbb R^n,
 \end{eqnarray}
where $\cM=\left\{A_i\right\}_{i\in {\cal I}}$ (${\cal I}$ is a set of indices) is a set of $n\times n$ real matrices that describes the system,
and $x_0$ is the initial state. So, a trajectory is uniquely
defined by an \emph{initial condition} $x_0$ and a \emph{switching signal} $\sigma(t):\mathbb N \rightarrow {\cal I},$ which represents the sequence of matrices ruling the dynamics.\\
Switching systems constitute an important
family of hybrid systems. They have been the subject of a great research
effort in recent years (see \cite{liberzon-switching,jungers_lncis,Shorten05stabilitycriteria,sun-ge} for
general introductions and applications in systems and control).  Of particular importance in many applications is the \emph{worst-case stability} of these systems.  It is well known that their stability is ruled by the
\emph{joint spectral radius} of the set of matrices $\cM,$ which we now define:

For each $k=1,2,\ldots $, consider the set ${\mathcal
P}_k(\cM)$ of all possible products of length $k$ whose factors are
elements of $\cM$, that is $\P_k(\cM)=\{A_{i_1}\ldots 
A_{i_k}\ | \ i_1,\ldots,i_k \in {\cal I}\}$ and set
$\PF = \bigcup\limits_{k \ge 1} \P_k(\cM)$
to be the \emph{multiplicative
semigroup} associated with $\cM$.

\begin{definition}[Joint Spectral Radius -- jsr \cite{RoSt60}] If
$\|\cdot\|$ is any matrix norm on $\mathbb{R}^{\nxn}$, let
$$\widehat\rho_k(\cM):=\sup_{P\in \P_k(\cM)} \|P\|^{1/k},
\qquad k\in \N.$$  
The {\em joint spectral radius} (jsr) of $\cM$ is defined as
\begin{equation}\label{eq:JSR}
\widehat\rho(\cM)=\lim_{k\rightarrow \infty} \
\widehat\rho_k(\cM).
\end{equation}
\end{definition}

The joint spectral radius does not depend on the matrix norm chosen
thanks to the equivalence between matrix norms in finite dimensional
spaces.  It is a natural generalization of the classical notion of spectral radius of a matrix to a set of matrices. It is also known to be very hard to compute \cite{tsitsiklis97lyapunov}. Yet, it is of high importance in practice, because it characterizes the stability of switching systems:

\begin{thm}(e.g. \cite{jungers_lncis})
The system (\ref{switching-system}) converges to the origin for any initial point $x_0$ and any switching signal $\sigma$ if and only
if $$\rho(\cM)<1.$$
\end{thm}

One could come up with another natural generalization of the spectral radius to a set of matrices:
\begin{definition}[Generalized Spectral Radius -- gsr \cite{daub-lag}]\label{def:GSR}
Let $\rho(\cdot)$ denote the spectral radius of an $n\!
\times\! n $--matrix, consider
$\overline\rho_k(\cM):=\sup_{P\in \P_k(\cM)} \rho(P)^{1/k},
\qquad k\in \N$ and define the {\em generalized spectral radius} (gsr) of $\cM$ as
\begin{equation}\label{eq:GSR}
\overline\rho(\cM)=\limsup\limits_{k\rightarrow
\infty} \ \overline\rho_k(\cM).
\end{equation}
\end{definition}

In the case of bounded sets $\cM$ Daubechies and Lagarias conjectured the equality of gsr and jsr, which was then proven by Berger and Wang \cite{berger-wang}. From now on we shall restrict ourselves to finite sets of matrices, and we call \emph{the joint spectral radius} of a finite set $\cM$ the quantity

$$
\rho(\cM) \triangleq \widehat\rho(\cM)=\overline\rho(\cM).
$$
It is not difficult to see that for any bounded set of matrices,
\begin{equation}\label{eq-sandwich}
\forall\, k,p \quad \overline\rho(\cM)_k \leq \rho (\cM) \leq \widehat\rho_p(\cM).
\end{equation}

The joint spectral radius has found a great number of applications in several fields of Engineering, Mathematics, and Computer Science.  See \cite{jungers_lncis} for diverse applications.  Recently, several methods from very different flavors have been proposed to provide approximations of this quantity \cite{MM11,Shorten05stabilitycriteria,valcher-positive,ajprhscc11}. In particular, there has been a growing interest in the study of the so-called \emph{extremal norms} (also called Barabanov norms) \cite{morris-barabanov12,chitouretal12,jungersprotasov09}.  We now detail this idea.

\subsection{Extremal norms} 

An important characterization of the joint spectral radius $\rho(\mathcal{M})$ of
a matrix set is the generalization of the following well-known formula,
$\rho(A)=\inf_{\|\cdot\|\in {\rm Op}}\|A\|$,
where ${\rm Op}$ denotes the set of all possible operator norms.

In order to state this characterization, we define the norm of the set
$\mathcal{M}=\{A_i\}_{i\in \mathcal{I}}$ as
$$
\|\mathcal{M}\|=\hat{\rho }_1(\mathcal{M})=\max_{i\in \mathcal{I}}\|A_i\|.
$$
\begin{proposition}[see e.g. \cite{jungers_lncis}]
The joint spectral radius of a bounded set $\mathcal{M}$
of $n\times n$-matrices is given by
\begin{equation}
\rho(\mathcal{M})=\inf_{\|\cdot\|\in {\rm Op}}\|\mathcal{M}\|.
\label{c.s.r.}
\end{equation}
\label{proposition_csr}
\end{proposition}
We remark that the right-hand side of (\ref{c.s.r.}) is often referred to as
the common spectral radius of the set $\mathcal{M}$.

Given a set $\mathcal{M}$, it is important to establish if
the infimum in (\ref{c.s.r.}) is indeed a minimum.
We have the following definition.
\begin{definition}[Extremal norm]
A norm \ $\|\cdot\|_*$ satisfying the condition
\begin{equation}
\|\mathcal{M}\|_*=\rho(\mathcal{M})
\label{6}
\end{equation}
is said to be extremal for the set $\mathcal{M}$.
\label{def_extremality}
\end{definition}
For sets of matrices with jsr equal to one, a sufficient condition for the existence of an extremal norm is that
the product semigroup generated by $\cM$ is bounded and the set is irreducible.

An extremal norm is a norm whose unit ball is invariant under the action of the matrices in the set.  This confers a practical interest to the analysis of such objects: if one manages to generate a norm whose unit ball is an invariant set, then this directly allows to bound the jsr: $\widehat \rho \leq 1.$ \begin{remark}We consider the particular question $\widehat \rho \leq 1$ without loss of generality because one can divide all the matrices by a constant $r$ and ask whether the newly obtained set has joint spectral radius smaller than one.  By homogeneity of the Definition \ref{def:GSR}, the minimal $r$ such that the answer is yes provides the actual value of the joint spectral radius. This is the strategy that we will follow in this paper.\end{remark}  Methods based on the generation of an invariant polytope have been proposed in \cite{protasov1,GuWiZe05,GuZe09,GuZe08}.  A great advantage of these methods is that they can stop in finite time, and provide an exact computation of the jsr of the set of matrices.  Indeed, if, besides the knowledge of an invariant polytope, one knows a product whose spectral radius is equal to one, then (by Eq. (\ref{eq-sandwich})) one has a proof that the exact value of the jsr is one.

\subsection{Invariant cones}
Along an other line of research, methods have been proposed, which base their analysis on the existence of a common invariant cone for the set of matrices, or which first transform the set of matrices so that the new matrices share an invariant cone (\cite{blondel-kron,jungers-protasov-blondel09,parrilo-jadbabaie,protasov-laa-2010,jungers-laa12}). In this paper, by \emph{cone} we mean a convex, closed, pointed, and nondegenerate cone.
Let $\cK \subset \rn$ be a cone with the apex at the origin (see e.g. \cite{BoyVan04}). Any
such cone defines a partial order in $\rn$: we write $x \geq_\cK y \;(x >_\cK y)$ for $x-y \in \cK \; ( x-y \in\inter\cK )$.

The cone $\cK$ is an invariant cone for the matrix $A$ if $A\cK \subset \cK$. In this case we say that $A$ is nonnegative and write $A \geq_\cK 0$. If $\cK$ is invariant for all matrices of some set $\cM$, then it is said to be an invariant cone for that set.
%normally we do not need these results anymore
%\begin{lemma}
%Given a convex set $\cS\subset\cK$ such that $\forall x \in \cK$ with $x\neq 0$ the quantity $f_\cS(x):=\inf\{\lambda\, \big|\, x\in\lambda\cS\}$ is finite and strictly positive, then it is possible to construct a norm $|\cdot|_\cS$ on $\rn$ such that $|x|_\cS=f_\cS(x),\, \forall x \in \cK$ .
%\end{lemma}
%\begin{proof}{Sketch}
%Consider as unit ball of the norm $|\cdot|_\cS$ the set $co\{\cS \cup -\cS\}$.
%\end{proof}
%
%\begin{theorem}{\cite[Corollary 1]{jungers-protasov-blondel07}}
%\end{theorem}
We now present a simple algebraic operation, that enables one to construct a new set of matrices without essentially changing the JSR. Moreover it turns out that for any set of matrices $\cm,$ the newly constructed set has an invariant cone: the set $\mathcal S^n_{+}$ of semidefinite positive matrices.
\begin{definition}Given a vector $v\in \co^n,$ its \emph{semidefinite lifting} is defined as $$\widetilde v=\RE(v\cdot v^*).$$ Also, the \emph{semidefinite lifting} of a matrix is the linear operator:
$$\widetilde A: \cS^n_+ \rightarrow \cS^n_+, \quad \widetilde A \widetilde v = \widetilde{Av}. $$
\end{definition}
We will note $a\preceq b$ (resp. $a\succeq b$) for $a \leq_{ \cS^n_+ }b$ (resp. $a \geq_{ \cS^n_+ }b$).

The following important result relates the spectral radii of a set of matrices and
the corresponding lifted set \cite{jungers-protasov-blondel09}.
\begin{thm} \label{th:equivalence}
The following equality holds true for any bounded set of matrices $\cM$:
$$\rho(\widetilde \cm)=\rho(\cm)^2.$$
\end{thm}

Thus, if one wants to compute the joint spectral radius of $\cm,$ he can as well compute it for the transformed set $\widetilde \cm.$
\\
Our goal is to exploit the invariant cone property of the lifted set
in order to compute the joint spectral radius by determining an extremal norm for
$\widetilde \cM.$  As we will see, it will be sufficient to find an invariant set inside the invariant cone.  A convenient way of describing such a set is given in the following definition.  It is reminiscent of the convex hull operation, but, besides being defined by a set of points, it also depends on the geometry of an underlying cone.

\begin{definition}[Conical convex hull and Conitope] Given a cone $K$ and a (possibly infinite) set of vectors $U\subset K$, 
its \emph{conical convex hull} is defined as 
$$\coni(U) \triangleq \{x \in K: \exists\, y\in \conv{(U)}: x\leq_\cK y \}.$$

We say that a set $S\subset K$  is a \emph{conitope} if there is a 
\emph{finite set} $U = \{ u_j \}_{j=1}^{\ell}$ such that
\begin{eqnarray}\nonumber S&= &\conitope{(U)} \triangleq \{x \in K: \exists\, \lambda_1,\dots, \lambda_\ell,\lambda_i\geq 0, \sum\lambda_i=1:\\ &&
 x\preceq \lambda_1 u_1 +\cdots + \lambda_\ell u_\ell \}. \end{eqnarray}
\end{definition}
Conitopes can in turn define a norm on the corresponding cone in the following way.

%\begin{definition} Given a set of vectors $U$ inside a cone $K$ such that $\conv{U}\cap \inter K \neq \emptyset,$ one can define the corresponding \emph{conitope norm} $$ |\cdot|_U:K \rightarrow \re,\quad |x|_U = \inf\{\lambda>0:\ x/ \lambda\in \conitope(U)\}. $$
%\end{definition}

\begin{lemma} \label{lem-valid-norm}
Let $\cK$ be a cone. Let $U$ be a set of points in $\cK,$ such that $\conv{(U)}\cap \inter \cK \neq \emptyset.$  Then, for any $x\in \cK,$ $|x|_U \triangleq \inf\{\lambda>0:\ x/ \lambda\in \conitope(U)\}$ is a valid norm on $K.$ We call such norms \emph{conitope norms.}
\end{lemma}
We present a last definition which characterizes minimal sets of vertices defining a conitope.
\begin{definition}[Essential system]
We say that $U\subset K$ is an \emph{essential system of vertices} if for any $U'\subset U,\, U'\neq U,$ $$\coni(U')\neq \coni(U). $$
\end{definition}
%\begin{definition}{Lifted polytope anti-norm}
%\end{definition}
\section{Plan}
 In this work, we show how to combine the two above mentioned approaches: our goal is to find conitope norms which are extremal for a given set of matrices.  We will show that this combination allows to derive innovative efficient algorithms for computing the joint spectral radius.  As a byproduct, the same ideas also provide an approximation algorithm for the joint spectral subradius of a set of matrices.  However, we focus here on the joint spectral radius computation for the sake of clarity.
 In the next section, we present the algorithms that unify the two approaches.  Then, in Section \ref{sec-thms}, we prove that the algorithms are correct and converge to the required value.  In Section \ref{sec-finiteness}, we exhibit a large family of sets of matrices for which our algorithms provably converge in finite time. In section \ref{section-examples} we give evidence for the good performance of our method on numerical examples. In the conclusion, we point out open questions and generalizations that naturally come up from our methods.

\section{The algorithms}

\subsection{Computing the joint spectral radius} 

We start by giving the following useful definition.
\begin{definition}[Spectrum-maximizing product]
If $\cM$ is a bounded set of complex $n\times n$-matrices, any
matrix $A\in \cP_{t}(\cM)$ satisfying 
\begin{equation}
\rho(A)^{1/t}=\bar{\rho}_{t}(\cM)=\rho(\cM).
\label{eq:smp}
\end{equation}
for some $t\geq 1$ is called a
spectrum-maximizing product (SMP) for $\cM$.
\label{definitionSMP}
\end{definition}

The procedure of Algorithm $1$ can be summarized in three steps, which
are analogous to those used in the classical polytope 
algorithms (\cite{GuWiZe05,GuZe08,GuZe09}).
The first step is to look for a candidate spectrum maximizing product
$A \in \cP_{t}(\cM)$ for some $t\ge 1$ 
and to scale the set as
$$\cM' = \cM/\rho(A)^{1/t},$$ 
which satisfies $\rho(\cM') \ge 1$.

The second step is to look for an invariant convex set for 
the lifted set $\widetilde {\cM'}$, i.e., the unit ball of 
a conic extremal norm.

If the procedure succeeds, then we can conclude that 
$\rho(\cM') = 1$, that is, the jsr of our initial set of matrices is actually equal to $\rho(A)^{1/t}.$

Some escaping criteria to detect that $A$ is not an SMP, can be used in 
analogy to those used for classical polytope algorithms (see e.g. \cite{ciconeetal09}, \cite{guglielmiprotasov}).

Algorithm $2$ builds an extremal conitope in a dynamic manner, without reinitializing the algorithm, and pruning the products to consider thanks to the partial order implied by the cone $\cS^n_+.$

\subsection{Two algorithms} 
 We present here our two algorithms. They both generate an upper bound $Y$ and a lower bound $C$ which converge towards the joint spectral radius.

We denote as $\mbox{L.E.}(A)$ the leading eigenvector of a matrix $A$, implicitely assuming this
is unique. In the case it is not unique we will specify this.

\textbf{Algorithm 1: the conitope method}

{\bf\sc Algorithm 1:}

\begin{itemize}
\item {\bf  Preprocessing:}
\item Find a product $A$ of length $t$ such that $\rho(A)^{1/t}$ is large (if needed, perform an exhaustive search for increasingly long products)
\item $C=\rho(A)^{1/t}$
\item $v_0:= \mbox{L.E.}(A) $
\item $U_0:=\widetilde v_0 $
\item $\cm':= \cm/C$  %$l(A)$ is the length of the product $A$
\begin{itemize}
\item while $\mathcal \rmj{\mbox{int }S_{n+}\cap \mbox{span}(U_0)= \emptyset}$
\begin{itemize}
\item Let $U_0$ be an essential system of vertices of $\conitope (U_0\cup \widetilde{\cm'} U_0)$
\end{itemize}
\item $U := U_0$
\end{itemize}
\item {\bf Main loop}
\begin{itemize}
\item Repeat until $B=1$
\begin{itemize}
\item $W:= \widetilde{\cm'} U$
\item $B:= \max{\{|w|_U:w\in W\}},\, Y=B\cdot C$
\item Let $U$ be an essential system of vertices of $\conitope (U\cup W)$
\item $D:= \max{\{|w|_U:w\in U_0\}}$
\item If $D<1$
\begin{itemize}
\item $A$ is not an SMP: go back to preprocessing
\end{itemize}
\end{itemize}
\item Return $\conitope (U)$ as an invariant conitope.
\end{itemize}
\item End
\end{itemize}

\textbf{Algorithm 2: the dynamical procedure}

We now present a modification of the algorithm which is allowed by the fact that the matrices share a common invariant cone.  The idea is that the cone induces a partial order relation between points, which allows us to disgard some of them, if they are dominated by another point.\\
This algorithm is a dynamical version of the previous one: while trying to prove that a certain product $A$ is an SMP, at the same time it tries and finds another better product $A'.$ Once it finds such a product, the algorithm simply continues its computation, except that the stored points in $V$ are scaled as if the algorithm had started with $A'.$

{\bf\sc Algorithm 2:}
\begin{itemize}
\item {\bf  Preprocessing:}
\item $C=\max{\{\rho(A):A\in \cm\}}$
\item $\cm = \cm/C$

\item $U_0:= I $

%\begin{itemize}
%\item while  $\conv{U_0}\cap \inter K = \emptyset$
%\begin{itemize}
%\item $U_0=U_0\cup \widetilde \cm U_0$
%\end{itemize}
\item $U := U_0$
%\end{itemize}
\item {\bf Main loop}
\begin{itemize}
\item Repeat until $B= 1$
\begin{itemize}
\item $W:= \widetilde \cm U$
\item $B:= \max{\{|w|_U:w\in W\}},$ $Y=B\cdot C$
\item Let $U$ be an essential system of vertices of $\conitope (U\cup W)$

\item \label{item-subproduct} $L=\max_{A}{\{\rho(A)^{1/l(A)}:A\in f(z), z\in W\}}$  \emph{(  $f(z)$ is the set of subproducts of the product that maps $U_0$ on $z.$  $l(A)$ is the length of the product $A$.)}
% \comng{\ This is unclear to me; what is $v_0$?}
    \item if $L>1$
\begin{itemize}
\item $C:=L\cdot C$
\item $\widetilde \cm := \widetilde \cm /L$
\item For all $u\in U:$ $u:=u/(C^{l(u)})$ \emph{($l(u)$ is the length of the product 
$\widetilde A$ such that $\widetilde A U_0 = u.$)}
 
\end{itemize}
\end{itemize}
\item Return $\conitope (U)$ as an invariant conitope.
\end{itemize}
\end{itemize}

\begin{remark}
The algorithms differ by two facts: first, in the initialization part, the first one starts with a leading eigenvector of a candidate SMP, and the second one with the identity matrix; second (and mainly), in the first algorithm, one reinitializes the conitope, by restarting from the new leading eigenvector, while the second algorithm keeps \emph{all the previous vertices} and continues with them.  These two modifications are independent from each other, but the main idea is that in the first algorithm, starting with a leading eigenvector enables us to stop in finite time (indeed, it is known that the leading eigenvectors of the permutations of the SMP are all on the boundary of any extremal norm).  For the second algorithm, we start with the identity matrix because we do not really hope to converge in finite time, but we are more interested in obtaining rapidly converging bounds than by a finite time termination.

Finally, in algorithm $1$ as presented here, we do not search for a better spectral radius among all the subproducts in $f(z)$ for all $z\in W,$ like we do in algorithm $2.$  However, one could perform this search as well in Algorithm $1$, which would allow to potentially find a better SMP without having to resort to an exhaustive search when coming back to the preprocessing phase.
\end{remark}

\section{Bounds and convergence results}\label{sec-thms}
%
%\begin{thm}
%Let $\cm$ be a set of matrices.  If $\cm$ admits an extremal bounded complex polytope with $k$ vertices $\{v_1,\dots v_k\},$ then $\widetilde \cm$ admits an extremal lifted polytope norm with vertices $\{\widetilde v_1,\dots,\widetilde v_k\}.$
%
%Moreover, if Algorithm BCP finds the bounded complex polytope in $t$ steps, then Alorithm 1 finds the lifted polytope in $t$ steps.
%\comrj{I have a weaker version describing an invariant set for $\widetilde \cm,$ but not a conitope}
%\end{thm}
In this section, we show that our methods are correct, which means that they provide valid lower and upper bounds on the joint spectral radius, and that moreover, these bounds converge towards the actual value asymptotically.
\begin{proposition}\label{prop-upperbound}
Let $\widetilde \cm$ be a set of matrices that share an invariant cone $\cK.$ Let $U$ be a set of points in $\cK,$ such that $\conv{(U)}\cap \inter \cK \neq \emptyset.$  Then, we have the following valid upper bound on the joint spectral radius: \begin{equation} \label{eq-prop-upperbound} \rho(\widetilde \cm)\leq \widehat \rho_U(\widetilde \cm),\end{equation}
where \begin{equation} \label{eq-prop-upperbound-def}\widehat \rho_U(\widetilde \cm)=\sup_{u\in U}{\sup_{\widetilde A\in \widetilde  \cm}{\{|\widetilde Au|_U\}}} .\end{equation}
\end{proposition}

\begin{proof} By Equation (\ref{eq:GSR}), for any $\epsilon >0,$ there exists a $t$ and a product $\widetilde A\in \cP(\widetilde  \cm)$ such that $\rho(\widetilde A)\geq (\rho(\widetilde \cm)-\epsilon)^t$.  
Since the matrices in $\widetilde \cm$ share an invariant cone $\cK,$ by the generalized Perron-Frobenius Theorem, there exists a vector $u\in \cK$ such that $\widetilde Au = (\rho(\widetilde \cm)-\epsilon)^t u. $ Let us take the vector $u$ such that $|u|_U=1.$\\ 
Now, $\widetilde Au \in  \rho_U(\widetilde \cm)^tU $ (as implied by Eq. (\ref{eq-prop-upperbound-def})) together with $\widetilde Au=(\rho-\epsilon)^tu$ imply $$ \rho_U(\widetilde \cm) \geq \rho-\epsilon. $$ since this must hold for any $\epsilon$ we obtain (\ref{eq-prop-upperbound}).
\end{proof}

Since algorithm $1$ starts with a set of vectors $U_0,$ which are rank one matrices, the set $U$ actually contains only rank one matrices during the whole run of the algorithm.  Hence,
 it is not clear that one obtains for sure a valid norm (as required to go out of the while loop in the algorithm).  The next proposition shows how to deal with this problem:
\begin{proposition}\label{prop-subspace}
If the matrices in $ \cm$ do not have a common invariant real subspace, then Algorithm $1$ generates in finite time a set $U$ such that  $\conv(U) \cap \inter \mathcal S^n_{+}\neq \emptyset.$\\ Moreover,
if the algorithm does not generate such a set $U$ after $n$ iterations, then (with the notations of Algorithm $1$) $S=\mbox{span } \PF(\real{(v_0)})$ is an invariant real subspace, and one can project the matrices on $S$ and on $S^{\perp}$ and iterate the algorithm. 
\end{proposition}
\begin{proof}
By contraposition, suppose that $\conv \{({\cP(\widetilde \cm) }\widetilde v_0)\}\cap \inter{\cS_{n+}}=\emptyset.$ This implies that there exists a real vector $x$ such that $$\forall A \in \PF, x^* A (\widetilde v_0) A^*x =0.$$  Then, letting $$v_0=v_1+iv_2,\quad v_i\in \re^n,$$ we have that \begin{equation}\label{eq-prop-subspace}\forall A \in \PF, x^* A (v_1v_1^*+v_2v_2^*) A^*x =0,\end{equation} which implies that $$x^*Av_iv_i^* A^*x=0,\quad i=1,2$$ and $\PF v_i$ is a non trivial invariant subspace for $\cm.$ 

Moreover, Equation (\ref{eq-prop-subspace}) holds if and only if it holds for all $A\in {\mathcal P}_k(\cM),$ $k=1,2,\ldots n.$ Thus, if it still holds at step $n,$ one can project $\cm$ on (for instance) $\PF v_1$ and on its orthogonal space, and reiterate the algorithm on the two sets of smaller matrices.
\end{proof}

In the proof of our main theorem below, we also need the following technical proposition:

\begin{proposition}\label{lem-az-z}\cite{elsner-generalized}
Let $||\cdot||$ be a matrix norm in $\cn$ induced by the vector norm $|\cdot|.$  There is an absolute constant $c(n)>0$ such that for all $z\in\cn,\, |z|=1,$ and all $A\in \mathbb C^n,\, ||A||\leq 1,$ there is an eigenvalue $\lambda$ of $A$ such that
$$|1-\lambda|\leq c |Az-z|^{1/n}.
$$
\end{proposition}

%
%\begin{lemma}\label{lem}\cite{elsner-generalized}
%If a set of matrices $\cm$ is irreducible, then, for any two nonzero vertices $x_1,x_2,$ there exists a product $A\in \PF$ such that $x_1^*Ax_2>0.$ \end{lemma}

\begin{thm}\label{thm-conv}
Both algorithms $1$ and $2$ are exact, i.e., both the upper bound $Y=C*B$ and the lower bound $C$ asymptotically converge towards the joint spectral radius.
\end{thm}
 The proof can be found in the extended version of this paper \footnote{\emph{Lifted polytope methods
for the asymptotic analysis of matrix semigroups},Preprint, 2012, by the same authors.}.
\begin{remark}
If the set $\cm$ is not irreducible, then we can still use Algorithm $2$, but we cannot converge for sure towards an invariant conitope, because such a conitope might not exist.  However, the lower bound is still valid, and still converges towards to JSR.  For the upper bound, one has to resort to the classical estimate $\rho \leq \sup_{A\in \cm^t} (||A||^{1/t}).$
\end{remark}
%
%\begin{thm}
%
%If Algorithm BCP finds the bounded complex polytope in $t$ steps, then Algorithm 2 finds the lifted polytope in less than $t$ steps. \comrj{I'm not sure about this one, but since we build on the previous norm it should be the case? or perhaps under some conditions that we might find.}
%\end{thm}

%
%\begin{thm}
%
%If Algorithm BCP finds the bounded complex polytope in $t$ steps, then Algorithm 2 finds the lifted polytope in less than $t$ steps. \comrj{I'm not sure about this one, but since we build on the previous norm it should be the case? or perhaps under some conditions that we might find.}
%\end{thm}
\section{Finiteness results}\label{sec-finiteness}

\subsection{Preliminary results: asymptotic simplicity and extremal polytope norms}

We denote here the normalized set 
$$
{\mathcal{M}}^{\p} = \frac{\mathcal{M}}{\rho(\mathcal{M})}.
$$
Now we report the existence results proved in \cite{GuWiZe05}, which involve the
normalized trajectories 
$$
\mathcal{T}[{\mathcal{M}}^{\p},x] = \{x\} \cup \{Px \mid P\in \cP(\cM^{\p})\}
$$
where $\cP(\cM) \triangleq \bigcup\limits_{k \ge 0} \cP_{k} (\cM)$.
First we observe that all the cyclic permutations of a product
$\bar{P}$ have the same eigenvalues with the same multiplicities.
Thus, if $\bar{P}=A_{i_{k^{\p}}}\ldots A_{i_1}$ is an SMP for
a set $\mathcal{M}$, then each of its cyclic permutations
$$
A_{i_s}\ldots A_{i_1}A_{i_{k^{\p}}}\ldots A_{i_{s+1}},
\ \ \ \ \ s=1,\ldots ,k^{\p}-1,
$$
still is an SMP for $\mathcal{M}$.

\begin{definition}[Asymptotically simple set]
A nondefective bounded set $\mathcal{M}$ of complex $n\times n$-matrices is
called \emph{asymptotically simple} if it has a minimal SMP $\bar{P}$ with only
one leading eigenvector (modulo scalar nonzero factors) such that the set
$X$ of the leading eigenvectors of $\mathcal{M}$ is equal to the set of
the leading eigenvectors of $\bar{P}$ and of its cyclic permutations.
\label{definition_as}
\end{definition}

We remark that the asympotic simplicity of $\mathcal{M}$ means that the set
$X$ of the leading eigenvectors of $\mathcal{M}$ is {\em $\mathcal{M}$-cyclic},
i.e., for any pair $(x,y)\in X \times X$ there exist $\alpha, \beta \in \C$ with
$$
|\alpha|\cdot |\beta|=1
$$
and two normalized products $P^{\p},Q^{\p}\in \P(\mathcal{M}^{\p})$ such that
$$
y=\alpha P^{\p} x \ \ \  {\rm and}  \ \ \ x=\beta Q^{\p} y.
$$

\begin{thm}\cite{GuWiZe05}\label{theo-as}
Assume that a finite set $\mathcal{M}$ of complex
$n\times n$-matrices is nondefective and asymptotically simple.
Moreover, let $x\neq 0$ be a leading eigenvector of $\mathcal{M}$ and assume that
${\rm span}\Big(\mathcal{T}[\mathcal{M}^{\p},x]\Big)=\C^n$. Then the set
\begin{equation}
\partial \mathcal{S}[\mathcal{M}^{\p},x] \bigcap \mathcal{T}[\mathcal{M}^{\p},x]
\label{intersect}
\end{equation}
is finite (modulo scalar factors of unitary modulus). As a consequence,
there exist a finite number of normalized products $P^{\p}_1,
\ldots ,P^{\p}_s\in \P(\cM')$ such that
\begin{equation}
\mathcal{S}[\mathcal{M}^{\p},x]={\rm absco}\Big(\{x,P^{\p}_1 x,\ldots ,
P^{\p}_s x\}\Big),
\label{eq:bcp}
\end{equation}
so that $\mathcal{S}[\mathcal{M}^{\p},x]$ is an invariant balanced complex polytope.
\label{theoremGWZ}
\end{thm}

\subsection{Asymptotic simplicity and extremal conitope norms}
% \comrj{we should perhaps harmonize with the text, but I preferred not to modify the text since it is in Nicola's hands.}
We assert that if a set of matrices is asymptotically simple, not only it admits an extremal polytope norm as claimed by Theorem \ref{theo-as}, but its lifted set also admits an \emph{extremal conitope norm.}  The proofs are to be found in the extended version of this paper.  We start with a technical lemma.
%In the following, by $\conv{(U)}$ we mean the convex hull of the adherence, that is the closure points, of $U.$ 
\begin{lemma}\label{lem-conv2zero}
 If a sequence of vectors $s_i \in \cn$ converges to zero, the set
$\coni(\{ \tilde s_i\})$ is a conitope. 
%
%\comng{By consistency we should assume $s_i \in \C^n$ unless we want to \\
%restrict our analysis to the real case. This point is also \\ encountered before} 
\end{lemma}
%The proof of this lemma is to be found in Appendix 

\begin{thm}\label{thm-as}
If a finite set of matrices $\cM$ is irreducible and asymptotically simple, then its lifted set $\widetilde \cM$ admits an extremal conitope.  Moreover, there exists such an extremal conitope whose all vertices are rank one matrices.
\end{thm}

\begin{remark} \rm \ \\
Note that if one applies the polytope algorithm in \cite{GuWiZe05,GuZe09} and then lifts the vertices
$\{ v_i \}$ and considers the corresponding conitope, one does not obtain in general an invariant set 
for the lifted set of matrices. This is a key point in order to understand that the two algorithms are
fundamentally different (see the following Example \ref{ex-non-inv} in Section \ref{sec:numer}). 

\end{remark}

\section{Numerical examples}\label{section-examples}
\label{sec:numer}

We provide here some illustrative examples demonstrating the effectiveness of the
methods proposed in this article.

\begin{ex} \rm
Consider the set of matrices:
$$\cm= \left \{  
\left(
\begin{array}{rrrr}
0&    -1   &     1     &    1\\
1&     0   &     0     &    0\\
0&    -1   &     0     &    0\\
1&    -1   &    -1     &    0\\ 
\end{array}
\right)
         ,\
\left(
\begin{array}{rrrr}
 0&   -1   &     1     &    0\\
-1&   -1   &     1     &    1\\
-1&    0   &     0     &    0\\
-1&   -1   &     0     &   -1
\end{array}
\right)
\right \}.
$$

This set of matrices has a joint spectral radius $$ \rho(\cm) = \rho(A_2) \approx 1.779.$$  

The classical BCP algorithm terminates in $4$ steps and finds a BCP which is the convex hull of $16$ vertices.  
Algorithm $1$ terminates after $2$ steps and finds an invariant conitope with only $7$ vertices.
\end{ex}
%\begin{ex}
%An example where the conitope algorithm works worse (larger number of iteration, larger number of vertices, or both?). If we don't find this example, we can conjecture/prove that it always works better.
%\end{ex}

\begin{ex}\rm
Our next example presents a set of complex matrices.
{\small
\begin{eqnarray}\nonumber \cm&=& \left \{
\left(
\begin{array}{rrrr}
-1+i&    -i   &     -1+i     \\
0&     1   &     -1-i     \\
1+i&    -i   &     -1-i     \\
\end{array}
\right)\right.
         ,\ \\ \nonumber
&& \left. \left(
\begin{array}{rrrr}
i&    -1-i   &     -1     \\
1-i&     -1+i   &    i     \\
-1+i&    1+i   &     1+i     \\
\end{array}
\right)
\right \}.
\end{eqnarray}}

In this case $ \rho(\cm) = \rho(A_1A_1A_2A_1A_2) \approx 2.2401.$

The classical BCP algorithm requires $20$ steps to find a BCP with $65$ essential vertices.
Algorithm $1$ terminates after $8$ steps and the invariant conitope has only $10$ essential vertices.
\end{ex}
\begin{ex}\label{ex-non-inv}  \rm
The following set of matrices has an invariant polytope $\mbox{conv}{\{v_i\}},$ but it appears that simply lifting its vertices does not provide an invariant conitope: $\conitope (\{\widetilde v_i\})$ is not invariant for $\widetilde \cm.$ 
$\mathcal{M} = \{ A_1, A_2 \}$ with
$$
A_1= \left( \begin{array}{rrr} 0 & 1 &  1 \\  1 & 0 & 0 \\  0 & -1 & 0 \end{array} \right), \qquad 
A_2= \left( \begin{array}{rrr} 0 & 1 &  0 \\ -1 & 0 & 1 \\ -1 &  0 & 0 \end{array} \right).
$$
The spectrum maximizing product is $P = A_1 A_2$ and has a real leading eigenvector.
An extremal real centrally symmetric polytope has been computed by the algorithm described 
in \cite{GuZe08} and gives $6$ vertices:
$$
v_1 = x, \quad  v_2 = A_1^\p v_1, \quad v_3 = A_2^\p v_1, $$
$$ 
v_4 = A_2^\p v_2, \quad v_5 = A_2^\p v_3, \quad v_6 = A_1^\p v_5, 
$$
where $A_i^\p = A_i/\rho(P)$, $i=1,2$.

The simple algebraic test that $\widetilde w = \widetilde A_2^\p \widetilde v_2 $
does not belong to the conitope shows that computing the set of lifted vertices 
yields a conitope which is not invariant. 
\end{ex}
%\section{The joint spectral subradius} Very recently, a algorithm similar to Algorithms $1$ and $2$ has been proposed, in order to compute the \emph{joint spectral subradius} of a set of matrices $\cM.$

\section{Conclusion}
In this paper we introduced a new method for deciding the stability of switching systems.  By combining the lifting technique and the extremal polytope technique, we have shown that one can in some cases outperform the other algorithms in the literature.  We have also provided a sufficient condition for the existence of an extremal norm for the lifted matrices (``an extremal conitope'').

A huge advantage of our algorithms is that they allow to disgard in a systematic way many products in the semigroup, thanks to the partial order relation induced by the semidefinite cone. This can greatly reduce the computation time in some situations.  Another advantage is that the algorithms can terminate in finite time, and provide a proof that a particular product is actually an SMP, by exhibiting a convex set (the extremal conitope) which is a valid extremal norm.
The same ideas apply for the \emph{joint spectral subradius}, and provide the first efficient stopping criterion that allows to compute exactly the joint spectral subradius for general sets of matrices (i.e. with negative and positive entries).  We delay for further work the empirical study of these methods for the computation of the joint spectral subradius.\\ Also, we leave for further research the question of whether similar methods could be applied in order to approximate \emph{other joint spectral characteristics,} like the so-called $p$-radius, or the Lyapunov exponent.  Indeed, it has been shown recently that other methods based on semidefinite programming are useful for computing these quantities as well (e.g. \cite{jungers-protasov-pradius,protasov-jungers-lyap}).
 
The main computational burden, at every step, is in the computation of the quantity $|Av|_V,$ i.e., the semidefinite program allowing to compute the conitope norm of the new vertices.  If the matrices do not share a priori a common invariant cone, one must work in the lifted space in order to have an invariant cone (that is, $\mathcal S_{n+}$), but then one has to resort to semidefinite programming to represent the conitope on a computer.  We leave it as an open question whether it is possible to apply similar ideas to general matrices with a simpler cone, on which the computations are less expensive.

\section{Acknowledgements}
This work was carried out while RJ was visiting the University of L'Aquila.  This author wants to thank the University for its outstanding hospitality.

A Matlab implementation of the methods developed here is available in the JSR Toolbox \cite{jsr-toolbox}.

\appendix{Proof of Theorem \ref{thm-conv}\\
\begin{proof} In the following, $\rho$ denotes $\rho(\widetilde \cm)=\rho(\cm)^2.$

\textbf{Lower bound}  With the notations of Algorithm $1,$ let us call $\cm'$ the scaled set of matrices: $$ \cm'=\cm/C.$$
We want to prove that the lower bound $C$ converges towards the actual value of the Joint Spectral Radius.  Thus, for Algorithm $1,$ it suffices to show that if $\rho(\cm')>1$ (i.e. if $A$ is not an SMP), the main loop terminates in finite time, and the algorithm will hence look for longer products by returning to the preprocessing phase;  since we have the well known alternative definition (\ref{eq:GSR}) for the JSR, the lower bound will converge towards $\rho.$\\ 
From Proposition \ref{prop-subspace}, we can suppose that we have $n$ vectors $v_1,\dots,v_n\in \cP(\cm') v_0$ linearly independent. 

Since $\rho(\cm')>1,$ it is possible to find arbitrary large vectors satisfying this property.  In particular, they can be such that $$\sum(\widetilde v_i)/n \succeq \widetilde v_0.$$ 

Hence, the algorithm will terminate, because the last equation implies that $U_0 \subset \conitope(U)$ at some time, which is exactly the termination condition in Algorithm $1.$

For algorithm $2,$ recall that for any $t$ and any norm $|\cdot|,$
$$\rho(\widetilde\cm)^t \leq \sup \{|\widetilde A u| : |u|=1, \widetilde A\in \widetilde\cM^t\}.$$  Thus, there exists a vector $u\in U_0,$ an infinite product $\dots \widetilde A_t \dots\widetilde  A_1,$ and a sequence $t_{i_1}, t_{i_2},\dots$ such that $\forall i, |\widetilde A_{t_{i}}\dots\widetilde  A_1u|\geq \rho^{t_i}.$  Let us now consider the normalized sequence $u_t'=u_t/|u_t|,$ $u_t=\widetilde A_t \dots\widetilde  A_1 u.$ This is an infinite sequence on the boundary of the unit ball of the norm.  By compactness of this boundary, there exists a subsequence $s_i$ such that the $u_{s_i}'$ converge to a vector $y : |y|=1.$ So for any $\epsilon$, there must exist $s_j$ such that for all $s_i>s_j,$
$$|u_{s_j}'-u_{s_i}'|<\epsilon.$$\\  Setting $z=u'_{s_j}, \widetilde B_{i}=\widetilde A_{s_i}\dots \widetilde A_{s_j+1}(|u_{s_j}|/|u_{s_i}|),$ we get matrices $\widetilde B_{i}\in \P(\widetilde \cM)$ such that

$$ |B_{i}z-z|<\epsilon, $$ and we can conclude by Proposition \ref{lem-az-z} that $B_{i}$ has an eigenvalue $\lambda$ such that $|1-\lambda|=O(\epsilon^{1/n}), $ which implies that
$$\rho(\widetilde A_{s_i}\dots \widetilde A_{s_j+1})^{1/(s_i-s_j)}\rightarrow \rho. $$ 
As a conclusion, Algorithm $2,$ which checks the spectral radius of all the products of the type $\widetilde A_{t_i}\dots \widetilde A_{t_j},$ will find candidates SMP whose averaged spectral radius tends to $\rho.$
%Let $v_t= A_{i_t} \dots A_{i_1}v_0$ be this particular sequence: $|v_t|\geq k [)^*$ for some $k,\epsilon>0.$  Then, there must exist a sequence $B_{i_t}\dots B_{i_0}v_0\in V_t $ such that $B_{i_t}\dots B_{i_0}v_0\geq k (1+\epsilon)^*.$ This implies \cite[Lemma 2.1]{jungers_lncis} that there exists some subproduct $B_{i_s}\dots B_{i_{s'}}$ such that $\rho(B_{i_s}\dots B_{i_{s'}})>1.$
%
%Let us suppose that we start algorithm $2$ with the conitope $V_0.$
%Fix an arbitrary $\lambda >0.$
%Since $\rho(\cm)>1,$ there must exist a product and a $t$ $A\in \cm^*,$ such that $Av\geq \lambda v,$ where $v$ is the leading eigenvector of $A$ (whose conitope norm $|v|_{V_0}=1$). Now, there exist $\lambda_i>0, \sum{\lambda_i}=1,$ such that $v\leq \sum{\lambda_i v_i}.$ Then, $$\lambda v< Av < A(\sum{\lambda_i v_i}), $$ and this implies that there exists an index $i$ such that $Av_i>\lambda v.$  Thus, if $v\in \inter \cK,$ $Av_i>v_i.$ Now, there exists a $B,$ $Bv_i\in W(t),\, Bv_i>Av_i,$ and then $Bv_i>v_i.$

\textbf{Upper bound}
Again from Proposition \ref{prop-subspace}, we can suppose that $\cm$ has no real invariant subspace.  This implies \cite[Theorem 2.1]{jungers_lncis} that the semigroup $\cP\left( (\widetilde \cm/\rho) \right)$ is bounded.  

Therefore, denoting by $C_t$ the lower bound $C$ at step $t,$ we have $U_t=(\widetilde \cm/C_t)^*U_0$ tends to $U_{\infty}=\cP(\widetilde \cm/\rho) U_0$ in Hausdorff distance.  

Hence the quantity $\sup_{\widetilde A \in \widetilde \cm,  u \in U_t}{|\widetilde A u |_{U_t}}$ tends towards  $$\sup_{\widetilde A \in \widetilde \cm, u \in U_\infty}{|\widetilde A u|_{U_\infty}}=\rho(\widetilde \cm).$$
\end{proof}

\appendix{Proof of Lemma \ref{lem-conv2zero}\\
\begin{proof}

Define $W=span(\{s_i\}_{i>0}),$ $k=dim(W),$ and take a set $\{r_1,\dots ,r_k\}\subset \{s_i\}_{i>0}$ 
of full rank.
Then, all the $s_i$  can be written as $\sum_{j=1,\dots, k}{\lambda_j r_j}.$

%Now, take an arbitrary element $s_i$ in the sequence.  Hence,
%
%$s_i=\sum_{j=1,\dots, k}{\lambda_j r_j}.$

  Take now a vector $s_i$ with $i$ large enough such that \begin{equation}\label{eq-rs}\{r_1,\dots, r_k\}\subset \{s_1,\dots, s_{i-1}\},\end{equation} and such that, denoting $s_i=\sum\limits_{j=1}^{k}{\lambda_j r_j},$ ${|\lambda_j|<1/k^2}.$ 
  
{\textbf{We claim that}}\begin{equation}\label{eq-lem-zerosequence}
s_is_i^* \preceq \sum\limits_{j=1}^{k}(r_jr_j^*)/k.
\end{equation} 
Indeed, 
taking an arbitrary vector $x\in \re^n$, we have 
\begin{eqnarray*}
x^*s_is_i^*x &=&
\sum\limits_{j_1,j_2=1}^{k}{(\lambda_{j_1}\lambda_{j_2}) (x^*r_{j_1}r_{j_2}^*x)}\\&\leq &\sum\limits_{j_1,j_2=1}^{k}{|\lambda_{j_1}||\lambda_{j_2}| |r_{j_1}^*x||r_{j_2}^*x|}\\ &\leq& k^2 (1/k^2)^2 \max_j{|r_j^*x|^2}\\&\leq & x^*\Big(\sum\limits_{j=1}^{k} (r_jr_j^*)/k\Big)x. \end{eqnarray*}
Since it is true for any $x\in \re^n,$ (\ref{eq-lem-zerosequence}) follows by definition.

As usual we indicate as $\widetilde s = \RE(s s^*)$.
Equations (\ref{eq-rs}) and (\ref{eq-lem-zerosequence}) together imply that 
$$ \widetilde s_i \in \conitope{(\widetilde s_1,\dots, \widetilde s_{i-1})},$$ and then, 
$$\conitope(\{\widetilde s_1,\dots,\widetilde s_i\})=\conitope(\{\widetilde s_1,\dots,\widetilde s_{i-1}\}).$$ 
Now, take $i$ such that $$\forall l>i, \exists\, \lambda_1,\dots, \lambda_k : |\lambda_j|<1/k^2, s_l=\sum\limits_{j=1}^{k}{\lambda_j r_j}.$$

Thus, for $i$ large enough we have $$\conitope(\{\widetilde s_1,\dots,\widetilde s_i\})
=\conitope(\{\widetilde s_1,\dots,\widetilde s_{i-1}\}),$$
which implies that the limit of this set actually is a conitope, since the limit is reached for a finite number of vertices.
\end{proof}

\appendix{Proof of Theorem \ref{thm-as}}\\
\begin{proof}
We assume without loss of generality that $\rho(\cM) = \rho(\widetilde \cM)=1.$ 
Let us denote $X$ the finite set of leading eigenvectors (modulo scalar nonzero factors) of $\cM.$  Take a particular $v\in X$ and denote $\widetilde V$ the set of trajectories starting at 
$\widetilde v = \RE(vv^*)$ under arbitrary products of matrices in $\widetilde\cM:$ 
$$\widetilde V=\{\widetilde A_{i_t}\dots \widetilde A_{i_1} (\widetilde v): \widetilde 
A_{i_j}\in \widetilde \cM\ \ \forall j\}.$$

 We now define \begin{equation}\label{eq-defS}S=\coni{(\widetilde V)}.%\{M\in \cS_+^n: \exists\, W\in \conv{(\widetilde V)}: M\preceq W\}.
 \end{equation}
% \comrj{We should slightly generalize the operation $\conitope{(V)}$ to infinite sets $V.$  However, we keep the definition of a conitope as the set obtained by taking the conitope operation of a finite set of vectors.  We could also rename the conitope operation the "`conitope transitive closure operation"', perhaps it is clearer?}  
 
Since $\widetilde \cM$ is irreducible, again by \cite[Theorem 2.1]{jungers_lncis}, the set $\widetilde V$ is bounded, and $S$ is an invariant set for $\widetilde\cM.$ The last follows by the fact that if $U$ is an invariant set, then $\coni{(U)}$ is also 
an invariant set.

We will prove that there exist a finite number of vectors $\widetilde v_1,\dots, \widetilde v_l\in \widetilde V$ such that $S=\conitope(\{\widetilde v_1,\dots,\widetilde  v_l\}),$ and so $S$ is an invariant conitope for $\widetilde \cM.$ 

Suppose the contrary.\\ 
Then, there must be an infinite set of vectors $\widetilde V'\subset \widetilde V$ such that \begin{equation}\label{eq-infinite-traj}\forall\, \widetilde v_i \in \widetilde V', \widetilde v_i \not \in \coni{(\widetilde V'\setminus \widetilde v_i)}.\end{equation}  
Moreover, by compactness of the set $S$ 
we can suppose that $\widetilde V'$ is a subset of a single trajectory of $\widetilde v,$ that is, there exist indices  $i_1,i_2,\dots$ such that the set $$\widetilde V'=\{\widetilde A_{i_t}\dots \widetilde A_{i_1}\widetilde v: t\geq 0\}$$ contains an infinite set satisfying Equation (\ref{eq-infinite-traj}). 
(For a detailed proof of this fact see for instance 
\cite[Lemma 5.19]{GuWiZe05}.)

We note $\widetilde v_t=\widetilde A_{i_t}\dots \widetilde A_{i_1}\widetilde v,$ and 
correspondingly $ v_t=    A_{i_t}\dots   A_{i_1} v.$ 
Since this set is infinite, there must be an index $t_0$ such that $v_{t_0}\in X, \ v_{t_0+1}\not \in X.$\\
%  We {\textbf{claim}} that for any $t> t_0,$ $\widetilde v_t\not \in X.$ Indeed, by asymptotic simplicity, if $\widetilde v_t \in X,$ there exists $P \in \widetilde \PF$ such that $P\widetilde v_t=\widetilde v_{t_0,}$ and that would imply in turn that $PA_{i_t}\dots A_{i_{t_0+2}}A_{i_{t_0+1}}v_{t_0}=v_{t_0},$ and hence, premultiplying the above equation by $A_{i_{t_0+1}},$ we obtain that $A_{i_{t_0+1}}PA_{i_t}\dots A_{i_{t_0+2}}v_{t_0+1}=v_{t_0+1}. $ This implies that $v_{t_0+1}\in X,$ a contradiction. \\

%Thus, there must
%Moreover, the sequence $\widetilde v_t=\widetilde A_t \dots \widetilde A_1 \widetilde v $ is bounded away from zero, because the set $S$ is not a conitope.  
Now, by compactness of $\coni({\widetilde V'})$, there must exist a subsequence of $\widetilde v_t$ that converges to a vector $\widetilde w.$ This vector $\widetilde w$ is different from zero because, by Equation (\ref{eq-infinite-traj}) and Lemma \ref{lem-conv2zero}, the sequence $\widetilde A_{i_{t}}\dots \widetilde A_{i_1}\widetilde v$ is bounded away from zero.  Then, since $$ \widetilde v_t = \RE(v_tv_t^*), $$ this must be the case also for the original (unlifted) matrices: hence $$ \exists\, t_0<t_1<t_2<\dots t_k\dots: A_{i_{t_k}}\dots A_{i_{t_{0}+2}}v_{t_0+1} \rightarrow w,$$ and then $w\in X. $ Now, by asymptotic simplicity, there exists a matrix $P\in \P(\cM)$ such that $Pw=v_{t_0}.$ This implies that $P A_{i_{t_k}}\dots  A_{i_{t_{0}+2}} v_{t_0+1}  \rightarrow v_{t_0},$ and hence, premultiplying the above equation by $A_{i_{t_0+1}},$ we obtain that $$ A_{i_{t_0+1}}P A_{i_{t_k}}\dots  A_{i_{t_{0}+2}} v_{t_0+1}  \rightarrow v_{t_0+1},$$ and $v_{t_0+1}$ is in $X,$ which gives a contradiction. 
\end{proof}
\end{document}